\documentclass[a4paper]{amsart}
\pdfoutput=1

\usepackage[utf8]{inputenc}
\usepackage{mathtools,amssymb}
\usepackage{hyperref}
\usepackage{xcolor,graphicx}
\usepackage{mathrsfs}

\mathtoolsset{showonlyrefs}

\graphicspath{{images/}}

\newtheorem{theorem}{Theorem}[section]
\newtheorem{lemma}[theorem]{Lemma}

\newtheorem{corollary}[theorem]{Corollary}

\title[Unique continuation of the normal operator and geophysics]{Unique continuation of the normal operator of the X-ray transform and applications in geophysics}
\keywords{Inverse problems, X-ray tomography, normal operator, unique continuation, theoretical seismology.}
\subjclass[2010]{86A22, 44A12, 46F12}

\author{Joonas Ilmavirta}
\thanks{Department of Mathematics and Statistics, University of Jyv\"askyl\"a, P.O. Box 35 (MaD) FI-40014 University of Jyv\"askyl\"a, Finland; \href{mailto:joonas.ilmavirta@jyu.fi}{joonas.ilmavirta@jyu.fi}}
\author{Keijo M\"onkk\"onen}
\thanks{Department of Mathematics and Statistics, University of Jyv\"askyl\"a, P.O. Box 35 (MaD) FI-40014 University of Jyv\"askyl\"a, Finland; \href{mailto:kematamo@student.jyu.fi}{kematamo@student.jyu.fi}}

\date{\today}

\newcommand{\C}{{\mathbb C}}
\newcommand{\R}{{\mathbb R}}
\newcommand{\Z}{{\mathbb Z}}
\newcommand{\N}{{\mathbb N}}

\newcommand{\A}{{\mathcal A}}

\newcommand{\eps}{\varepsilon}
\newcommand{\der}{{\mathrm d}}

\newcommand{\Kelvin}{K}
\newcommand{\riesz}{I_{\alpha}}
\newcommand{\xrt}{X}
\newcommand{\no}{N}

\newcommand{\schwartz}{\mathscr{S}}
\newcommand{\tempered}{\mathscr{S}^{\prime}}
\newcommand{\rapidly}{\mathscr{O}_C^{\prime}}
\newcommand{\slowly}{\mathscr{O}_M}

\newcommand{\csmooth}{\mathcal{D}}
\newcommand{\smooth}{\mathcal{E}}
\newcommand{\cdistr}{\mathcal{E}'}
\newcommand{\distr}{\mathcal{D}^{\prime}}
\newcommand{\dimens}{d}
\newcommand{\kernel}{h_{\alpha}} 

\newcommand{\abs}[1]{\left\lvert #1 \right\rvert}
\newcommand{\ip}[2]{\left\langle #1,#2 \right\rangle}
\DeclareMathOperator{\spt}{spt}
\DeclareMathOperator{\ch}{ch}

\newcommand{\NTR}[1]{}

\begin{document}

\maketitle

\begin{abstract} 
We show that the normal operator of the X-ray transform in $\mathbb{R}^d$, $d\geq 2$, has a unique continuation property in the class of compactly supported distributions. This immediately implies uniqueness for the X-ray tomography problem with partial data and generalizes some earlier results to higher dimensions. Our proof also gives a unique continuation property for certain Riesz potentials in the space of rapidly decreasing distributions. We present applications to local and global seismology. These include linearized travel time tomography with half-local data and global tomography based on shear wave splitting in a weakly anisotropic elastic medium. 
\end{abstract}
\NTR{Removed ``generalizes earlier results to lower regularity" from abstract since similar result was already proved in \cite{KKW-stability-of-interior-problems} and mentioned in \cite{KEQ-wavelet-methods-ROI-tomography} when $\dimens=2$. Added rapidly decreasing distributions to the abstract.}

\section{Introduction}

Linearized travel time tomography of shear waves reduces mathematically to a version of the X-ray tomography problem under a suitable model.
We are interested in shear wave splitting of waves travelling through the mantle, leading us to a partial data problem.
The partial data problem of the X-ray transform can then be reduced to a unique continuation problem of the normal operator of the X-ray transform. We study the unique continuation property of the normal operator mathematically and apply it to show that our partial data problems arising from geophysics have unique solutions.\NTR{Added a paragraph clarifying the relevance of our result to geophysics.}

Consider the following X-ray tomography problem with partial data. Assume we have a compactly supported function or distribution~$f$ on~$\R^{\dimens}$, $\dimens\geq 2$, and an open set $V\subset\R^{\dimens}$.
Suppose we only know the integrals of~$f$ over the lines through~$V$ and the values of~$f$ in~$V$.
Does this information determine~$f$ uniquely?
In terms of the X-ray transform $\xrt$, if $\xrt f(\gamma)=0$ for all lines~$\gamma$ intersecting~$V$ and $f|_V=0$, is it true that $f=0$? The answer is positive and even more is true.

The partial data problem can be recast into a unique continuation problem of the normal operator $\no=\xrt^*\xrt$ of the X-ray transform.
\NTR{Minor rephrasing.}
In other words, if $\no f|_V=0$ and $f|_V=0$, does it imply that $f=0$?
The answer is `yes', and we prove a stronger unique continuation property for~$N$ where we only require that~$\no f$ vanishes to infinite order at some point in~$V$.
The proof also applies to some Riesz potentials of rapidly decreasing distributions. As a corollary we get the uniqueness result for the X-ray tomography problem with partial data.\NTR{Corrected ``$\no$ vanishes to infinite order" to ``$\no f$ vanishes to infinite order". Added comment about rapidly decreasing distributions.}

It is well known that the partial data problem or region of interest (ROI) problem has important applications in medical imaging (see e.g.~\cite{KKW-stability-of-interior-problems, KEQ-wavelet-methods-ROI-tomography, NA-mathematics-computerized-tomography, YYW-interior-reconstruction-limited-angle-data, YW-compressed-interior-tomography}).
We introduce two possibly new applications in theoretical seismology.
Namely, we show that one can uniquely solve a linearized travel time problem with receivers only in a small open subset of the Earth's surface.
In addition, we describe how to use shear wave (S-wave) splitting measurements to determine the difference of the S-wave speeds.
See section \ref{sec:applications} for details on these applications.\NTR{Added references.}

Similar partial data results are known in~$\R^2$ for compactly supported smooth functions, compactly supported $L^1$-functions and compactly supported distributions~\cite{CNDK-solving-interior-problem-ct-with-apriori-knowledge, KKW-stability-of-interior-problems,KEQ-wavelet-methods-ROI-tomography}.
Our method of proof applies to all dimensions $\dimens\geq 2$.
An important novelty is in looking at the partial data result from the point of view of unique continuation of the normal operator.
The theorem can be seen as a complementary result to the Helgason support theorem (see lemma \ref{lemma:helgason}) where one requires that the lines do not intersect the set in question. Our result can also be seen as a unique continuation property for the inverse operator of the fractional Laplacian~$(-\Delta)^s$.\NTR{Minor modifications and added references.}

We present two alternative proofs for the partial data problem.
The first proof uses the unique continuation property of the normal operator of the X-ray transform.
The second proof is more direct and uses spherical symmetry.
However, both proofs rely on a similar idea, differentiation of an integral kernel and density of polynomials. We also present an alternative proof for the unique continuation of the Riesz potential which is based on unique continuation of the fractional Laplacian.\NTR{Added comment about the alternative proof for unique continuation of Riesz potential.}

\subsection{The main results}

Denote by $\csmooth(\R^{\dimens})$ the set of compactly supported smooth functions and by\NTR{Added preposition.} $\distr(\R^{\dimens})$ the space of all distributions in~$\R^{\dimens}$, $\dimens\geq 2$.
Also denote by $\cdistr(\R^{\dimens})$ the set of compactly supported distributions in~$\R^{\dimens}$.
Let $\alpha=\dimens-1$ or $\alpha\in\R\setminus\Z$ and $\alpha<\dimens$.
We define the Riesz potential $\riesz f=f\ast\kernel$ for $f\in\cdistr(\R^{\dimens})$ where $\kernel(x)=\abs{x}^{-\alpha}$ and the convolution is understood in the sense of distributions.
If $\alpha=\dimens-1$, then~$\riesz$ reduces to the normal operator of the X-ray transform up to a constant factor 2.
We say that $\riesz f$ vanishes to infinite order at a point $x_0$ if $\partial^{\beta}(\riesz f)(x_0)=0$ for all $\beta\in\N^\dimens$.
Our main result is the following (see also theorem~\ref{thm:alternativeconvolutionproof} and theorem~\ref{thm:alternativeproofrieszpotential}).\NTR{Added references for the alternative proofs.}

\begin{theorem}
\label{thm:maintheoremfornormaloperator}
Let $f\in\cdistr(\R^{\dimens})$, $V\subset\R^{\dimens}$ any nonempty open set and $x_0\in V$. If $f|_V=0$ and $\riesz f$ vanishes to infinite order at $x_0$, then $f=0$. In particular, this holds for the normal operator of the X-ray transform.
\end{theorem}

The condition $f|_V=0$ guarantees that $\riesz f$ is smooth in a neighborhood of~$x_0$.
The pointwise derivatives $\partial^{\beta}(\riesz f)(x_0)$ therefore exist, see the proof of theorem \ref{thm:maintheoremfornormaloperator} for details.
The condition of vanishing derivatives at a point only makes sense under the assumption that $f$ vanishes (or is smooth) in $V$.\NTR{Minor modifications and explanation why $\riesz f$ is smooth in $V$.}

Theorem~\ref{thm:maintheoremfornormaloperator} can be seen as a unique continuation property of the Riesz potential~$\riesz$.
The result resembles a strong unique continuation property but the roles in the decay conditions are interchanged.
As an immediate corollary we obtain the following partial data results for the X-ray tomography problem.
The first one is similar compared to the uniqueness results in~\cite{KKW-stability-of-interior-problems, KEQ-wavelet-methods-ROI-tomography}.
For the definition of the X-ray transform on distributions, see section~\ref{sec:integralgeometryanddistributions}.\NTR{Removed ``..is a generalization.." since similar results are proved in \cite{KKW-stability-of-interior-problems, KEQ-wavelet-methods-ROI-tomography} for $f\in\cdistr(\R^2)$.}

\begin{theorem}
\label{thm:partialdataproblem}
Let $V\subset\R^{\dimens}$ be any nonempty open set. If $f\in\cdistr(\R^{\dimens})$ satisfies $f|_V=0$ and $\xrt f$ vanishes on all lines that intersect $V$, then $f=0$.
\end{theorem}

\begin{corollary}
\label{cor:uniquenessinannulus}
Let $R>r>0$ and $f\in\cdistr(\R^{\dimens})$ such that $\spt(f)\subset \overline{B}(0, R)\setminus B(0, r)$.
If $\xrt f$ vanishes on all lines that intersect $B(0, r)$, then $f=0$.
\end{corollary}

\begin{corollary}
\label{cor:applicationstoseismicarrays}
Let $\Omega\subset\R^{\dimens}$ be a bounded, smooth and strictly convex set and $\Sigma\subset\partial\Omega$ any nonempty open subset of its boundary.
If $f\in\cdistr(\R^{\dimens})$ is supported in~$\overline{\Omega}$ and its X-ray transform vanishes on all lines that meet~$\Sigma$, then $f=0$.
\end{corollary}

Proofs of the theorems and corollaries can be found in section~\ref{subsec:proofsofmainresults} (see also the alternative proofs in section~\ref{sec:alternativeproofs}).
Some of our assumptions are crucial for the theorems to be true.
Theorem~\ref{thm:partialdataproblem} is clearly false if $\dimens=1$.
The function~$f$ cannot be determined from its integrals over the lines through the ROI only \cite{KEQ-wavelet-methods-ROI-tomography, NA-mathematics-computerized-tomography, SU:microlocal-analysis-integral-geometry}.
Thus one needs some information of~$f$ in the open set~$V$ which the lines all meet. Especially we need the assumption $f|_V=0$ when we use the Kelvin transform and density of polynomials. Our proof also exploits the assumption of compact support which is motivated by the physical setting and is needed to define the Riesz potential on distributions. However, one can relax that assumption to rapid decay at infinity (see theorem \ref{thm:alternativeconvolutionproof} and theorem \ref{thm:alternativeproofrieszpotential}). Theorem~\ref{thm:partialdataproblem} and corollaries~\ref{cor:uniquenessinannulus} and~\ref{cor:applicationstoseismicarrays} have important applications in theoretical seismology and medical imaging.
This is discussed in more depth in the next section.\NTR{Minor modifications, added references and mentioned why we need the assumption $f|_V=0$.} 

\subsection{Applications}
\label{sec:applications}

Our results have theoretical applications in seismology.
Applications include linearization of anisotropies in S-wave splitting and linearized travel time tomography. Even though there exist many different types of seismic data, we only use linearized travel time data without reflections in our models. For the following treatment of splitting of S-waves we refer to \cite{CL-decade-of-shear-wave-splitting, LS-shearwavesplitting-and-mantle-anisotropy, MP-wave-propagation-anisotropic-media, SA-seismic-anisotropy-and-mantle-deformation, SHE-introduction-to-seismology}.\NTR{Clarified what kind of seismic data we use in our models.}

In linear elasticity in~$\R^3$ there are three polarizations of seismic waves which correspond to the eigenvectors of the symmetric Christoffel matrix.
The eigenvalues correspond to wave speeds.
In the isotropic case the largest eigenvalue is simple with the eigenvector parallel to the direction of propagation, corresponding to a P-wave.
The other eigenvalue is degenerate with eigenvectors orthogonal to the P-wave polarization. These eigenvectors correspond to S-waves.
In anisotropic medium this degeneracy is typically lost and the degenerate S polarization splits to two quasi-S (qS) polarizations.
The data in the imaging method based on S-wave splitting is the arrival time difference between the two qS-waves.

One common type\NTR{Rephrased.} of anisotropy is hexagonally symmetric anisotropy.
This means that there is a preferred direction or a symmetry axis and the velocities vary only with the angle from the axis, i.e. there is rotational symmetry.
For example sedimentary layering and aligned crystals or cracks can cause hexagonal anisotropy.
If the seismic wavelength is substantially larger than the layer or crack spacing, then the material appears to be anisotropic~\cite{BA-long-waves-anisotropy-by-layering}.
The widely used one-dimensional Preliminary Reference Earth Model (PREM) indicates this kind of anisotropy between the depths 80--220 km in the upper mantle \cite{DA-PREM-model, SHE-introduction-to-seismology}.
In the PREM-model the symmetry axis is radial and all the physical parameters of the Earth depend only on the depth.
Anisotropies have also been observed in the shallow crust and in the inner core where the fastest direction is parallel to the rotation axis of the Earth \cite{CRE-anisotropy-inner-core, SHE-introduction-to-seismology}.

Our results pertain to so-called weak anisotropy, where we consider the anisotropy as a small perturbation to an isotropic reference model.
In the isotropic background model S-waves have a speed~$c_0(x)$ for all directions and polarizations.
When we add a small anisotropic perturbation, the speeds become\NTR{Fixed a typo.} $c_i(x, v)=c_0(x)+\delta c_i(x, v)$, $i=1, 2$.
Here $v\in S^2$ is the direction of propagation of the wave.
In the linearized regime $\abs{\delta c_i}\ll \abs{c_0}$ we have 
\begin{equation}
\frac{1}{c_i(x, v)}=\frac{1}{c_0(x)+\delta c_i(x, v)}\approx \frac{1}{c_0(x)}-\frac{\delta c_i(x, v)}{c_0^2(x)}.
\end{equation}
If we only measure small differences in the arrival times, our data is roughly
\begin{equation}
\label{eq:traveltimedifference}
\delta t\approx\int_{\gamma}\frac{\der s}{c_1(x, v)}-\int_{\gamma}\frac{\der s}{c_2(x, v)}\approx \int_{\gamma}\frac{\delta c_2(x, v)-\delta c_1(x, v)}{c_0^2(x)}\der s.
\end{equation}
Thus upon linearization, the data is the X-ray transform of $c_0^{-2}(\delta c_2-\delta c_1)$.
To simplify this problem, we assume the function to depend on~$x$ but not on~$v$.
If the splitting occurs in a layer near the surface (see figure~\ref{fig:datainsplitting}), we are in the setting of corollary~\ref{cor:uniquenessinannulus}.
The corollary implies that the linearized shear wave splitting data determines $\delta c_2-\delta c_1$ and thus $c_2-c_1$ uniquely in the outermost layer.\NTR{Emphasized that we uniquely recover the difference of anisotropies in the outermost layer.}

\begin{figure}[htp]
\centering
\includegraphics[height=7cm]{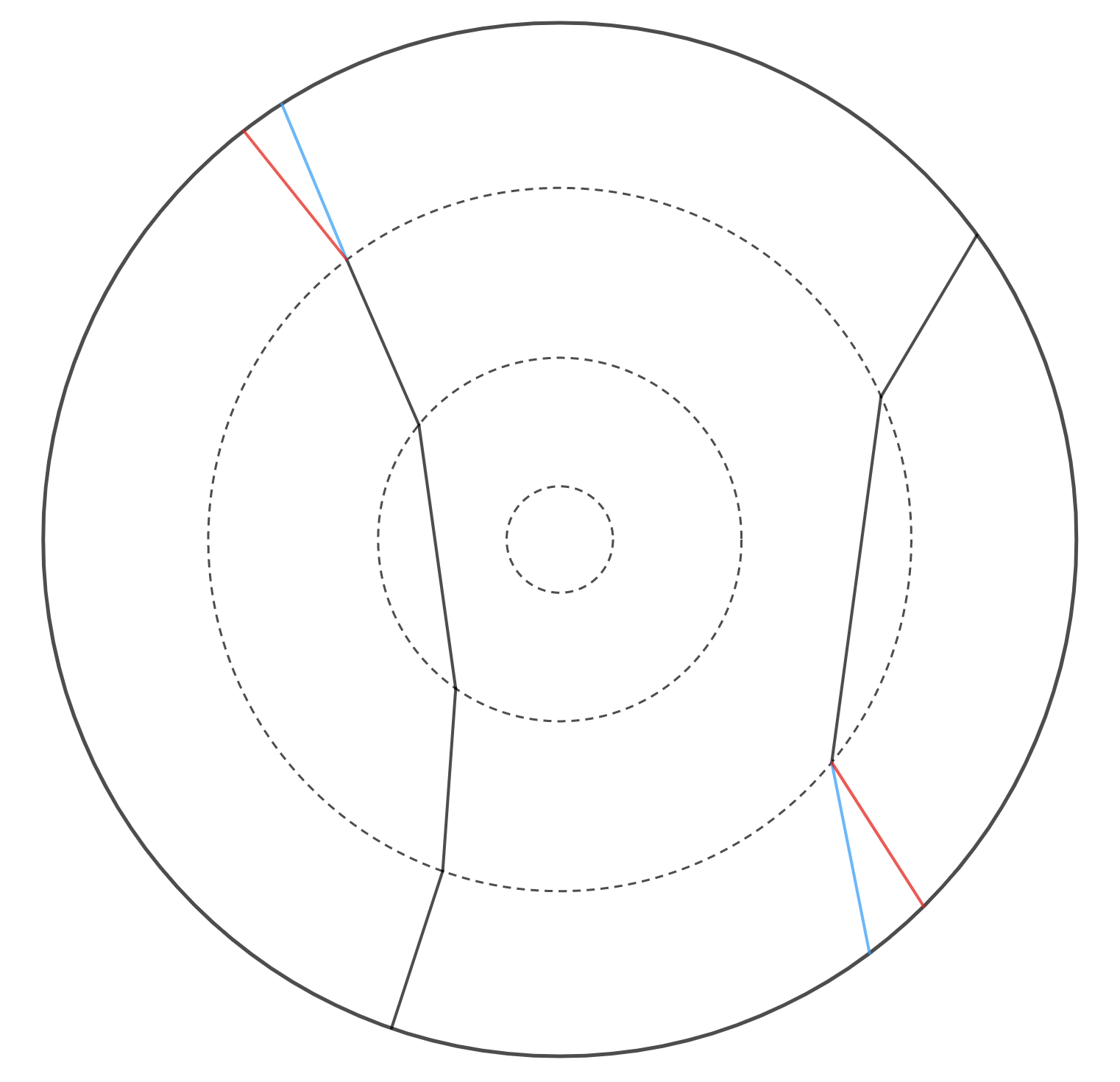}
\caption{A highly simplified picture of the setting in the linearized model. The splitting can occur at every interface but we only care about the splitting near the surface with smallest difference in the arrival times. There may exist different polarization states during the propagation of the initial wave, we only assume that the second to last part is an S-type wave. Our data consists purely of the branched parts of the waves.}
\label{fig:datainsplitting}
\end{figure}

Travel time tomography has a close relationship to the boundary rigidity problem where the aim is to reconstruct the metric of a manifold from boundary distance measurements \cite{SUVZ-travel-time-tomography, UZ-journey-to-the-center-of-the-earth}.
In seismology these distances correspond to travel times of seismic waves which are assumed to propagate along geodesics or straightest possible paths in the manifold.
This problem is highly nonlinear and difficult to solve in full generality.
Thus it is relevant to consider the first-order approximation and linearize the problem.
When we linearize the general travel time tomography problem assuming our manifold to be~$\R^\dimens$ and that the variations in the metric are conformally Euclidean, the geodesics become lines and the problem reduces to the X-ray tomography problem of a scalar function.

Linearized travel time tomography motivates the following application of observing earthquakes by seismic arrays on the surface of the Earth.
In the context of corollary~\ref{cor:applicationstoseismicarrays} one can ideally think that some open set of the surface is covered densely by seismometers (see figure~\ref{fig:seismicarrays}).
One detects earthquakes only in this set and measures travel times of seismic waves originating anywhere on the surface.
In geometrical terms, our geodesics have one endpoint in this open set and the other endpoint can freely vary.
In contrast to ``local data" where both endpoints are in the small set, we call this setting ``half-local data".
The interesting question then is whether this limited set of travel time data can determine the inner structure of the Earth uniquely.
When we do the usual conformal linearization in the Euclidean background, we end up with partial X-ray tomography problem of a scalar function.
Corollary~\ref{cor:applicationstoseismicarrays} then tells that in principle one can use these kind of seismic arrays to uniquely determine the conformal factor in the linearization.

\begin{figure}[htp]
\centering
\includegraphics[width=8cm]{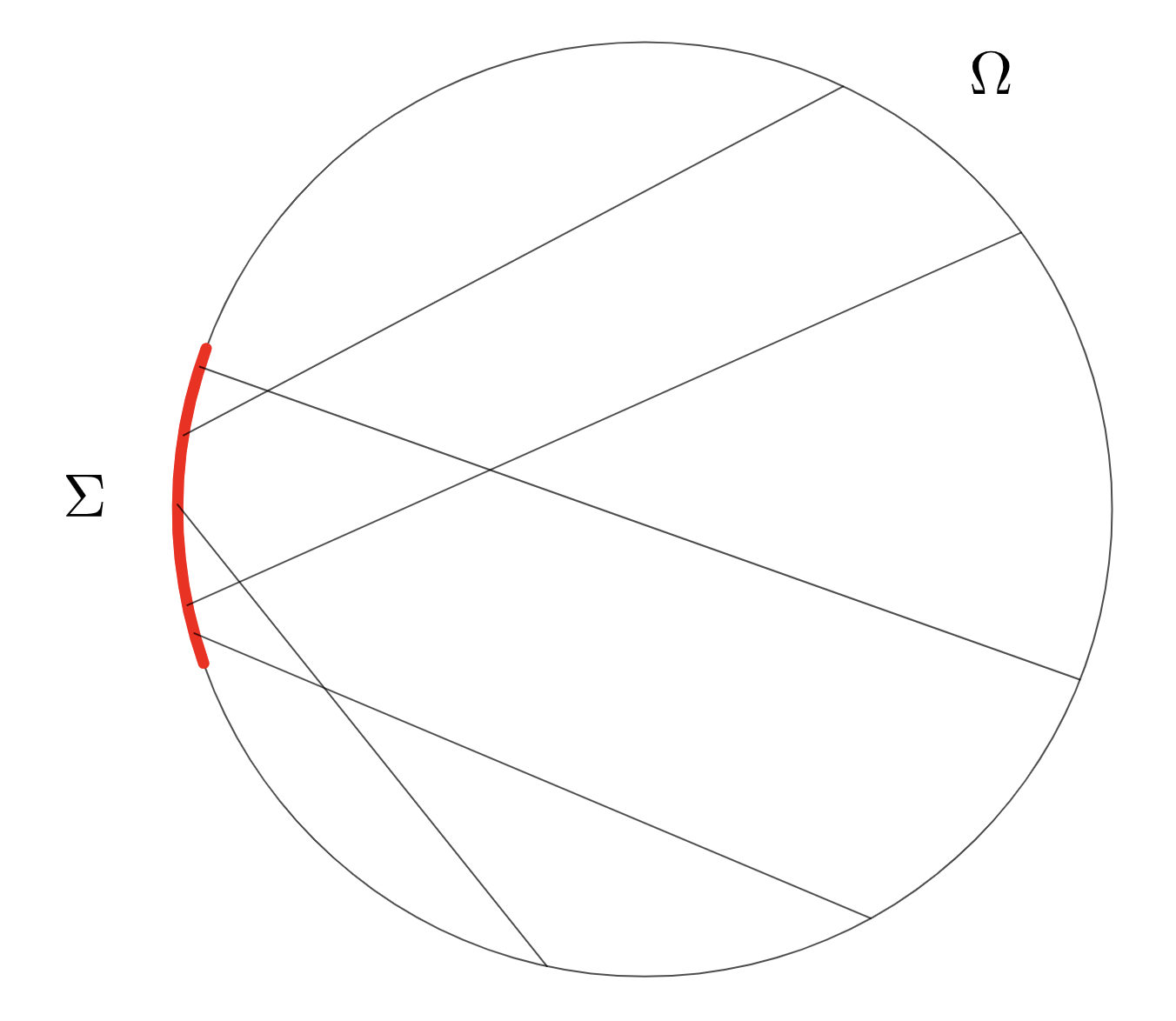}
\caption{The setting as in corollary~\ref{cor:applicationstoseismicarrays}. Here~$\Sigma$ (\textcolor{red}{thick}) represents the seismic array where one measures the travel times of seismic waves and~$\Omega$ represents the Earth.}
\label{fig:seismicarrays}
\end{figure}\NTR{Changed the symbol of the boundary set $\Gamma\rightarrow\Sigma$ to match with corollary \ref{cor:applicationstoseismicarrays}.}

In addition to theoretical seismology one important application is medical imaging, see~\cite{KKW-stability-of-interior-problems, KEQ-wavelet-methods-ROI-tomography, NA-mathematics-computerized-tomography, YYW-interior-reconstruction-limited-angle-data, YW-compressed-interior-tomography} and the references therein. Suppose we want to reconstruct a specific part of the human body, a region of interest (ROI).
Is it possible to reconstruct the image by shooting X-rays only through the ROI?
If this was possible it would be unnecessary to give a higher dose of X-rays to the patient and radiate regions outside the ROI which do not contribute significantly to the image.
We can interpret the function~$f$ in theorem~\ref{thm:partialdataproblem} as the attenuation of X-rays which to a good approximation travel along straight lines inside a body.
Somehow surprisingly theorem~\ref{thm:partialdataproblem} tells us that if we know the values of~$f$ in a small open set inside the ROI and the integrals of~$f$ over the lines going through the ROI, then~$f$ is uniquely determined everywhere (see figure~\ref{fig:roitomography}).
It is important to note that arbitrary attenuation cannot be determined from the line integrals only even in the ROI but one can always recover the singularities in the ROI~\cite{KEQ-wavelet-methods-ROI-tomography, NA-mathematics-computerized-tomography, SU:microlocal-analysis-integral-geometry}.\NTR{Added references.}

\begin{figure}[htp]
\centering
\includegraphics[width=7cm]{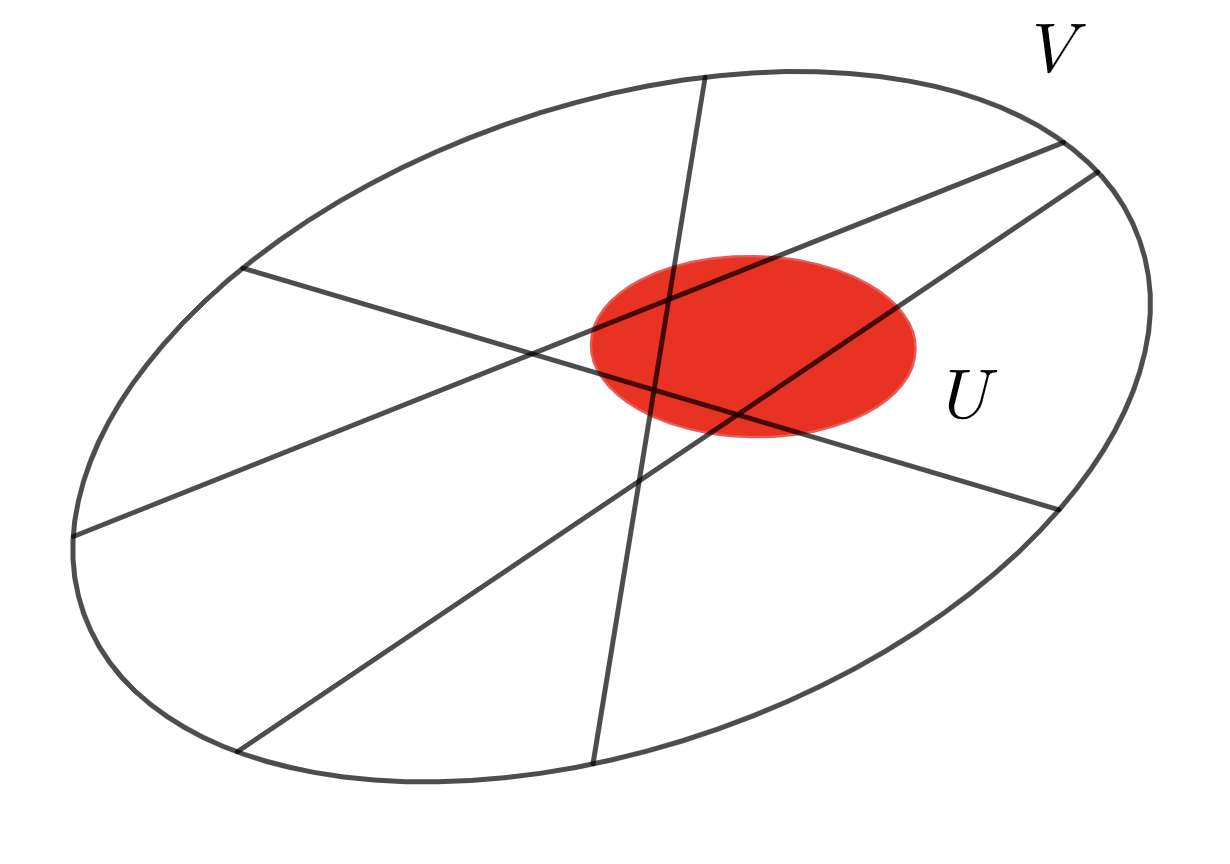}
\caption{Basic idea of ROI-tomography in the context of theorem~\ref{thm:partialdataproblem}. Here~$V$ is the region of interest and $U\subset V$ some open subset. If one knows the attenuation~$f$ in~$U$ and the integrals of~$f$ over the lines through~$V$, then one can construct~$f$ uniquely from the data.}
\label{fig:roitomography}
\end{figure}

\subsection{Related results}
The partial data problem for the X-ray transform has been solved earlier in~$\R^2$ under a variety of assumptions \cite{CNDK-solving-interior-problem-ct-with-apriori-knowledge, KKW-stability-of-interior-problems, KEQ-wavelet-methods-ROI-tomography, YYJW-high-order-TV-minimization}. The uniqueness result is known for $C_c^{\infty}$-functions and compactly supported $L^1$-functions if one assumes the knowledge of $f$ inside an open set in the ROI \cite{CNDK-solving-interior-problem-ct-with-apriori-knowledge, KEQ-wavelet-methods-ROI-tomography}. One also obtains uniqueness without knowing the exact values of~$f$ in the ROI; if $f$ is piecewise constant or piecewise polynomial in the ROI, then the X-ray data determines $f$ uniquely~\cite{KEQ-wavelet-methods-ROI-tomography, YYJW-high-order-TV-minimization}. If $f$ is polynomial in the ROI, then one obtains stability as well~\cite{KKW-stability-of-interior-problems}. Closest to our theorem is the uniqueness result in~\cite{KKW-stability-of-interior-problems} (see also \cite{KEQ-wavelet-methods-ROI-tomography} where the authors mention in the proof of lemma 2.4 that their method applies also to compactly supported distributions which are piecewise constant in the ROI). According to that result, if $f\in\cdistr(\R^2)$ integrates to zero over all lines intersecting $V$ and $f|_V$ is real analytic, then $f=0$.

Our result for the partial data problem uses stronger assumption $f|_V=0$. This assumption is needed so that the Kelvin transformed function will be compactly supported and we can use density of polynomials.
However, our theorem applies to any dimension $\dimens\geq 2$.
Another difference is in the point of view; we consider the normal operator and observe that the same result holds for a larger class of Riesz potentials. Also our alternative proofs (theorem \ref{thm:alternativeconvolutionproof} and theorem \ref{thm:alternativeproofrieszpotential}) imply uniqueness for the partial data problem without assumption of compact support, rapid decay at infinity is enough. We remark that the X-ray data alone does not uniquely determine the attenuation in general. One cannot even construct $C_c^{\infty}$-functions only from the integrals but one can always recover the singularities, which is equivalent with recovering the function up to a smooth error \cite{KEQ-wavelet-methods-ROI-tomography, NA-mathematics-computerized-tomography, SU:microlocal-analysis-integral-geometry}.\NTR{Added references and rewritten paragraph.}

Unlike in~\cite{CNDK-solving-interior-problem-ct-with-apriori-knowledge, KKW-stability-of-interior-problems} our method is very unstable and concrete reliable reconstructions are basically hopeless.
Our instability comes from the differentiation of the data and approximation of test functions by polynomials up to arbitrary order. However, our method of proof is not the only reason for instability.\NTR{Rephrased.} Instability is an intrinsic property of partial data problems. When we have limited X-ray data it is not guaranteed that we can see all the singularities of~$f$ from the data. Singularities which are invisible in the microlocal sense are related to the instability of inverting~$f$ from its limited X-ray data
\cite{KLM-on-local-tomography, MSU-geodesic-ray-transform-riemannian-surfaces, QU-singularities-x-ray-transform-limited-data, QU-introduction-x-ray-tomography-radon}.
See also \cite{KQ-microlocal-analysis-in-tomography, QU-artifacts-and-singularities-limited-tomography} for discussion of which part of the wave front set is visible in limited data tomography.
Even though our theorem loses stability it gives uniqueness which is relevant for applications.\NTR{Changed references since the authors in \cite{KEQ-wavelet-methods-ROI-tomography} mention that their uniqueness proof is not necessarily stable.}

Our theorem is related to travel time tomography and the inverse kinematic problem.
For a review of these, see \cite{SUVZ-travel-time-tomography, UZ-journey-to-the-center-of-the-earth} and also \cite{HE-inverse-kinematic-problem, WZ-kinematic-problem} for the original works by Herglotz, Wiechert and Zoeppritz.
Specifically our result is a contribution to local and global theoretical seismology (see section \ref{sec:applications}). For example one can uniquely determine the difference of the anisotropic perturbations of the S-wave speeds by measuring the arrival time differences of the split S-waves. From the point of view of ROI tomography these seismic applications are new to the best of our knowledge. 

It is also worth mentioning that our result is in a sense complementary to the famous support theorem by Helgason (see lemma~\ref{lemma:helgason}).
Helgason's theorem states that if $C\subset\R^\dimens$ is a convex compact set and $f\in\cdistr(\R^\dimens)$ such that $f|_C=0$ and the X-ray transform $\xrt f$ vanishes on all lines not meeting $C$, then $f=0$.
Compared to theorem~\ref{thm:partialdataproblem}, Helgason's result uses complementary data but gives the same conclusion. Helgason's theorem holds also for rapidly decreasing continuous functions; our partial data result is true for this function class as well (see section \ref{sec:alternativeproofs} and the discussion after theorem \ref{thm:alternativeconvolutionproof}).\NTR{Added comment about rapidly decreasing continuous functions.}

Our theorem has a connection to the fractional Laplacian $(-\Delta)^s$. The operator~$(-\Delta)^s$ can be defined in many equivalent ways and one way is to consider it as the inverse of a Riesz potential~\cite{KWA-ten-definitions-fractional-laplacian}.
In our notation $\riesz f=(-\Delta)^{-s}f$ where $s=(\dimens-\alpha)/2$ assuming $0<\alpha<\dimens$. For example from equation~\eqref{eq:inversionformulausingnormaloperator} we see that in Euclidean space the normal operator of the X-ray transform~$\no$ is the inverse of the fractional Laplacian~$(-\Delta)^{1/2}$.
Thus our result can be seen as a unique continuation property for the operator~$(-\Delta)^{-\delta/2}$ where $\delta$ is any positive non-integer or $\delta=1$.
There are several unique continuation results for the operator $(-\Delta)^s$ when $0<s<1$ and they have been recently used in fractional Calder\'on problems \cite{GRSU-fractional-calderon-single-measurement, GSU-calderon-problem-fractional-schrodinger, RI-liouville-riemann-integrals-potentials, RS-fractional-calderon-low-regularity-stability}. One version of our theorem can be proved using unique continuation of~$(-\Delta)^s$ (see theorem \ref{thm:alternativeproofrieszpotential}). The fractional Laplacian even admits a strong unique continuation property if one assumes more regularity from the function~\cite{FF-unique-continuation-fractional-ellliptic-equations,RU-unique-continuation-scrodinger-rough-potentials}.
Here ``strong" means that the function does not need to be zero in an open set, it only has to vanish to infinite order at some point.
Theorem \ref{thm:maintheoremfornormaloperator} has similar vanishing assumption for~$\riesz f$ instead of~$f$.
There are also (strong) unique continuation results for the higher order Laplacian~$(-\Delta)^{t}$ where~$t$ is a positive non-integer exponent~\cite{FF-unique-continuation-higher-laplacian, GR-fractional-laplacian-strong-unique-continuation, YA-higher-order-laplacian}.\NTR{Minor modifications and added reference to theorem \ref{thm:alternativeproofrieszpotential}.}

In Euclidean space one can reconstruct a compactly supported distribution uniquely from its X-ray transform~\cite{SU:microlocal-analysis-integral-geometry}.
There even exist explicit inversion formulas using the formal adjoint~$\xrt^*$ and the normal operator~$\no$.
It is also known that the X-ray transform is injective\NTR{Removed reference to solenoidal tensor fields, as that was more misleading than helpful in this context.} on compact simple Riemannian manifolds with boundary~\cite{IM:integral-geometry-review}.
Interesting injectivity results considering seismic applications have been obtained for conformally Euclidean metrics which satisfy the Herglotz condition~\cite{deI:abel-transforms-x-ray-tomography}.
See also how the length spectrum can be obtained from the Neumann spectrum of the Laplace-Beltrami operator or from the toroidal modes on these kind of manifolds in three dimensions~\cite{deIK-spectral-rigidity}.
This has a connection to the free oscillations of the Earth.

There are some partial data results for certain manifolds.
If $(M, g)$ is a two-dimensional compact simple Riemannian manifold with boundary and a real-analytic metric~$g$, then one can reconstruct $L^2$-functions locally from their geodesic X-ray transform~\cite{KRI-support-theorem-analytic-metric}.
In dimensions $\dimens\geq 3$ one can relax the analyticity condition to smoothness using a convexity assumption on the boundary~\cite{UV-local-geodesic-ray-transform-inversion}.
Furthermore one can even invert the X-ray transform locally in a stable way and obtain a reconstruction formula based on Neumann series.
Both of the results in \cite{KRI-support-theorem-analytic-metric, UV-local-geodesic-ray-transform-inversion} rely on microlocal analysis.
One can also locally invert, up to potential fields, tensors of order 1 and 2 near a strictly convex boundary point~\cite{SUV-local-invertibility-on-tensors}. We remark that there is a similar distinction between analyticity and smoothness for the injectivity of the weighted X-ray transform in Euclidean space.
When~$\dimens=2$ the analyticity of the weight is required for injectivity while in higher dimensions smoothness is enough~\cite{BO-nonuniqueness-weighted-radon-transform,BQ-support-theorems-analytic-radon-transforms, SU:microlocal-analysis-integral-geometry}. 

\subsection{Organization of the paper}
We begin our treatment by proving the main results in section~\ref{sec:mainresult}. We also discuss\NTR{Fixed typo.} the assumptions used in the results and applications. In section~\ref{sec:integralgeometryanddistributions} we recall some basic theory of distributions and integral geometry in~$\R^{\dimens}$. Section~\ref{sec:proofofpolynomiallemma} is devoted to the proof of lemma~\ref{lemma:generatingpolynomials} which says that one can express all the polynomials in a certain form as a finite linear combination of the derivatives of the kernel of the Riesz potential~$\riesz$. Section~\ref{sec:alternativeproofs} contains alternative proofs for theorem \ref{thm:maintheoremfornormaloperator} and theorem \ref{thm:partialdataproblem}.\NTR{Minor modifications.}

\subsection*{Acknowledgements}
J.I. was supported by the Academy of Finland (decision 295853) and K.M. was supported by Academy of Finland (Centre of Excellence in Inverse Modelling and Imaging, grant numbers 284715 and 309963).
We thank Maarten de Hoop and Todd Quinto for discussions. We also thank Mikko Salo for pointing out the connection between our result and the unique continuation of the fractional Laplacian. 
We are grateful to the anonymous referees for insightful remarks and suggestions.

\section{Proofs of the main results}
\label{sec:mainresult}

\subsection{An overview of the proof}

The rough idea of the proof of theorem~\ref{thm:maintheoremfornormaloperator} is the following.
We may assume that $x_0=0$. The function~$\riesz f$ is smooth in~$V$, and by assumption all of its derivatives vanish at the origin.
By a convolution argument these derivatives can be computed explicitly.
The vanishing of these derivatives amounts to~$f$ integrating to zero against a set of functions.
After a change of variables and suitable rescaling, one can use density of polynomials to show that this set is dense.
Therefore~$f$ has to vanish.

The proofs of the corollaries are more straightforward.
Detailed proofs of these main results are given in section~\ref{subsec:proofsofmainresults} below.
The reader who is not familiar with the theory of distributions and integral geometry can first read section~\ref{sec:integralgeometryanddistributions}.\NTR{Rephrased.}
See section~\ref{sec:alternativeproofs} for alternative proofs of theorems \ref{thm:maintheoremfornormaloperator} and~\ref{thm:partialdataproblem}.\NTR{Rephrased.}

\subsection{Auxiliary results}

In this section we give a few auxiliary results which are needed in our proofs. The first one is a known theorem in distribution theory.

\begin{lemma}[{\cite[p.160 Corollary 4]{TRE:topological-vector-spaces-distributions}}]
\label{lemma:polynomialsaredense}
Let $\Omega\subset\R^{\dimens}$ be an open set. Then the polynomials form a dense subspace of~$\smooth(\Omega)$.
\end{lemma}

Recall the kernel of the Riesz potential $\kernel(x)=\abs{x}^{-\alpha}$. The next lemma is proved in section~\ref{sec:proofofpolynomiallemma}.

\begin{lemma}
\label{lemma:generatingpolynomials}
If $\dimens\geq 2$ and $\alpha>\dimens-2$ or $\alpha\in\R\setminus\Z$, then for any polynomial~$p$ one can express the product $p(\Kelvin(x))\kernel(x)$ as a finite linear combination of derivatives of $\kernel$. Here $\Kelvin(x)=x\abs{x}^{-2}$ is the Kelvin transform.
\end{lemma}

We also need the following support theorem to prove corollary~\ref{cor:applicationstoseismicarrays}.
The proof can be found for example in \cite{HE:integral-geometry-radon-transforms, SU:microlocal-analysis-integral-geometry}.

\begin{lemma}[Helgason's support theorem]
\label{lemma:helgason}
Let~$C\subset\R^\dimens$ be a compact convex set and $f\in\cdistr(\R^\dimens)$. If $\xrt f$ vanishes on all lines not meeting $C$, then $\spt(f)\subset C$.
\end{lemma}

\subsection{Proofs of the results}
\label{subsec:proofsofmainresults}

Now we are ready to prove our main theorem and its corollaries. Let $\dimens\geq 2$. Recall the definition of the Riesz potential $\riesz f=f\ast \kernel$ for $f\in\cdistr(\R^\dimens)$ where $\alpha=\dimens-1$ or $\alpha\in\R\setminus\Z$ and $\alpha<\dimens$. The kernel $\kernel$ has an expression $\kernel(x)=\abs{x}^{-\alpha}$. We denote by $\Kelvin$ the Kelvin transform $K(x)=x\abs{x}^{-2}$. See section~\ref{sec:integralgeometryanddistributions} for basic results on distribution theory used in the proof.

\begin{proof}[Proof of theorem~\ref{thm:maintheoremfornormaloperator}]
We have to show that if $f\in\cdistr(\R^\dimens)$ and $V\subset\R^\dimens$ is any nonempty open set such that $f|_V=0$ and $\partial^{\beta}(\riesz f)(x_0)=0$ for some $x_0\in V$ and all $\beta\in\N^\dimens$, then $f=0$.
Because the problem is translation invariant we can assume that $x_0=0$.
Since~$f$ has compact support and it vanishes in a neighborhood of the origin, we have that $\spt(f)\subset A$ for some open annulus~$A$ centered at the origin.
Let $g\in\csmooth(\R^\dimens)$ be a symmetric smooth version of $\kernel$ such that $g|_A=\kernel|_A$.
Choosing small enough $\epsilon>0$ we have $\riesz f|_{B(0, \epsilon)}=(f\ast g)|_{B(0, \epsilon)}$ where $f\ast g\in \csmooth(\R^\dimens)$ by lemma~\ref{lemma:convolutionbetweencompactlysupporteddistributionadntestfunction}.
Since $\riesz f$ vanishes to infinite order at $0$ lemmas~\ref{lemma:convolutionbetweencompactlysupporteddistributionadntestfunction} and~\ref{lemma:derivativeofconvolution} give us $\partial^{\beta}(f\ast g)(0)=(f\ast(\partial^{\beta}g))(0)=\ip{f}{\tau_0\widetilde{\partial^{\beta}g}}=\ip{f}{\widetilde{\partial^{\beta}g}}=0$ for all multi-indices $\beta\in\N^\dimens$.
Since~$g$ is symmetric we get the condition $\ip{f}{ \partial^{\beta}g}=0$.

Let $\eta\in C^{\infty}_c(A)$ be such that $\eta=1$ in $\spt(f)$.
By lemma~\ref{lemma:testfunctionsagreeonsupport} and the definition of restriction~$f|_A$ we have $0=\ip{f}{\partial^{\beta}g}=\ip{f}{\eta\partial^{\beta}g}=\ip{f|_A}{\eta\partial^{\beta}g}$.
Since $g|_A=\kernel|_A$ by lemma~\ref{lemma:generatingpolynomials} we obtain all the polynomials~$p$ in the form $p(\Kelvin(x))\kernel(x)$ restricted to~$A$ by taking finite linear combinations of the derivatives of~$g$.
Using linearity we obtain $\ip{f|_A}{\eta \kernel (p\circ\Kelvin)}=0$ for all polynomials~$p$.
Taking the pullback we get $\ip{f|_A\circ\Kelvin}{\eta_1 p}=0$ where $\eta_1=((\eta\abs{J_{\Kelvin^{-1}}}^{-1})\circ\Kelvin)\kernel^{-1}$.
Let $\psi\in \smooth(\Kelvin^{-1}(A))$.
By lemma \ref{lemma:polynomialsaredense} there exists a sequence of polynomials~$p_k$ such that $p_k\to\psi$ in $\smooth(\Kelvin^{-1}(A))$.
This implies $\eta_1 p_k\to\eta_1\psi$ in $\smooth(\Kelvin^{-1}(A))$ because $\spt(\eta_1)\subset\subset\Kelvin^{-1}(A)$.
Since $f|_A\circ\Kelvin\in~\cdistr(\Kelvin^{-1}(A))$ by continuity $\ip{\eta_1 (f|_A\circ\Kelvin)}{\psi}=\ip{f|_A\circ\Kelvin}{\eta_1\psi}=0$, i.e. $\eta_1(f|_A\circ\Kelvin)=0$.
But now $\eta_1\neq 0$ in $\Kelvin^{-1}(\spt(f))=\spt(f|_A\circ\Kelvin)$ and hence $f|_A\circ\Kelvin=0$ by lemma~\ref{lemma:smoothfunctionnonzeroinsupport}.
Again using lemma~\ref{lemma:pullbackzerodistributionzero} we obtain $f|_A=0$ which implies $f=0$.\NTR{Added reference to lemma \ref{lemma:polynomialsaredense}.}
\end{proof}

As an immediate consequence we obtain the proofs for the X-ray tomography problem with partial data.

\begin{proof}[Proof of theorem~\ref{thm:partialdataproblem}]
We have to show that if $f\in\cdistr(\R^{\dimens})$ and $V\subset\R^{\dimens}$ is any nonempty open set such that $f|_V=0$ and~$\xrt f|_{\Gamma_V}=0$ where $\Gamma_V$ is the set of all lines that intersect~$V$, then $f=0$. We can assume that $V$ is a ball centered at the origin. Let $\varphi\in\csmooth(V)$.
From the definition of the normal operator of the X-ray transform we obtain $\ip{\no f}{ \varphi}=\ip{\xrt f}{\xrt\varphi}=0$ since $\xrt\varphi\in\csmooth(\Gamma_V)$.
Hence $\no f|_V=0$ and the claim follows from theorem~\ref{thm:maintheoremfornormaloperator} by taking $\alpha=\dimens-1$.
\NTR{Mentioned that $V$ can be assumed to be a ball.}
\end{proof}

\begin{proof}[Proof of corollary~\ref{cor:uniquenessinannulus}]
We have to show that if $R>r>0$ and $f\in\cdistr(\R^{\dimens})$ such that $\spt(f)\subset\overline{B}(0, R)\setminus B(0, r)$ and~$\xrt f$ vanishes on all lines that meet $B(0, r)$, then $f=0$. Take a nonempty open set $V\subset\subset B(0, r)$.
Then we have $f|_V=0$ and~$\xrt f$ vanishes on all lines that intersect~$V$.
Theorem~\ref{thm:partialdataproblem} implies that $f=0$.
\end{proof}

\begin{proof}[Proof of corollary~\ref{cor:applicationstoseismicarrays}]
Let $\Omega\subset\R^{\dimens}$ be a bounded, smooth and strictly convex set and $\Sigma\subset\partial\Omega$ nonempty open subset of the boundary.
We have to show that if $f\in\cdistr(\R^{\dimens})$ is supported in~$\overline{\Omega}$ and~$\xrt f$ vanishes on all lines that meet~$\Sigma$, then $f=0$.
We can assume that~$\Sigma$ is connected by passing to a connected component.
Denote by~$\ch(\Sigma)$ the convex hull of~$\Sigma$ (see figure \ref{fig:ideaoftheproof}).
By the Helgason support theorem (lemma~\ref{lemma:helgason}) the function~$f$ vanishes in~$\ch(\Sigma)$.
Take open set $V\subset\ch(\Sigma)$, $V\neq\varnothing$.
Then $f|_V=0$ and $\xrt f$ vanishes on all lines that intersect~$V$.
We can apply theorem~\ref{thm:partialdataproblem} to conclude that $f=0$.
\end{proof}

\begin{figure}[htp]
\centering
\includegraphics[width=7cm]{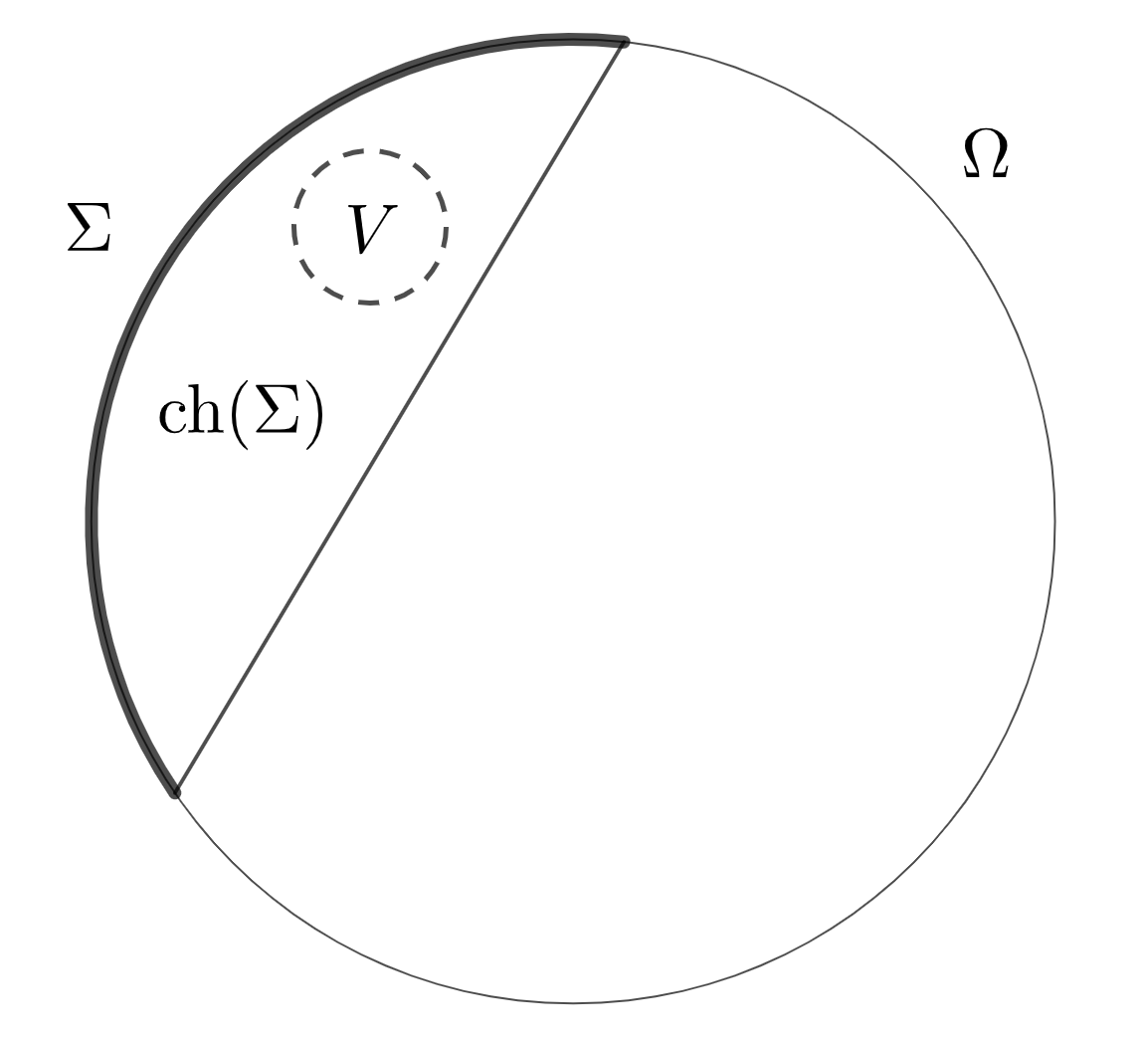}
\caption{Idea of the proof of corollary~\ref{cor:applicationstoseismicarrays}. Here~$\Sigma$ (thick arc) is a connected open subset of~$\partial\Omega$ and $\ch(\Sigma)$ (segment) its convex hull. Helgason's support theorem (lemma~\ref{lemma:helgason}) implies that~$f$ vanishes in $\ch(\Sigma)$ and then theorem~\ref{thm:partialdataproblem} is used for the dashed set~$V$ to conclude that $f=0$.}
\label{fig:ideaoftheproof}
\end{figure}

\subsection{Discussion of assumptions and methods}
\label{sec:discussionofassumptionsandmethods}

We assume that $f|_V=0$ so as to ensure that $\riesz f|_V$ is smooth and the differentiation makes sense.
For this purpose alone it would have been enough to assume that $f|_V$ is smooth.
However, if $f|_V$ is non-zero, our method of proof appears to become untractable. Especially the Kelvin transformed function $f\circ K$ is not compactly supported anymore and we can not use density of polynomials in the proof. If $f|_V$ is polynomial (or real analytic) and $\dimens=2$, the method of~\cite{KKW-stability-of-interior-problems} can be applied to prove the partial data result for the X-ray transform directly.
Our method has the additional freedom that the Riesz potential need not be exactly the normal operator and that the dimension is not restricted to two.
Moreover, in the physical application of shear wave splitting in the mantle, only the anisotropy in the mantle will matter and the perturbation can thus be taken to be supported outside the core.
\NTR{Added this paragraph.}

The assumptions in theorem~\ref{thm:maintheoremfornormaloperator} are not optimal. The assumption of compact support is needed to define the Riesz potential~$\riesz$ on distributions and is crucial in the proof when we use density of polynomials. However, compact support can be replaced with rapid decay at infinity (see theorem \ref{thm:alternativeconvolutionproof} and theorem~\ref{thm:alternativeproofrieszpotential}). Theorem \ref{thm:partialdataproblem} is clearly false if $d=1$. Also one cannot construct arbitrary $C_c^{\infty}$-functions from the integrals over the lines through the ROI only~\cite{KEQ-wavelet-methods-ROI-tomography, NA-mathematics-computerized-tomography, SU:microlocal-analysis-integral-geometry}. Therefore one needs some information of the function~$f$ in the open set~$V$; our method of proof especially requires the assumption $f|_V=0$. In corollary~\ref{cor:applicationstoseismicarrays} it is enough to assume that only the subset $\Sigma\subset\partial\Omega$ is strictly convex and the convex hull of the rest of the boundary does not cover all of~$\Sigma$. The constraint $\alpha<\dimens$ comes from the requirement that the kernel $\kernel$ determines a distribution. The other constraints for~$\alpha$ come from the proof of lemma \ref{lemma:generatingpolynomials}.
\NTR{Added references to alternative proofs and minor modifications.}

It would be interesting to know whether we could weaken the decay assumption in theorem \ref{thm:partialdataproblem} in the smooth case. Does there exist $f\in C^{\infty}(\R^\dimens)$ such that $f|_V=0$ and $\xrt f=0$ for all lines through~$V$ but~$f$ is not identically zero? By theorem~\ref{thm:alternativeconvolutionproof} the result in theorem \ref{thm:partialdataproblem} holds when $f$ decreases faster that any polynomial at infinity. There also exists a counterexample for the Helgason support theorem where the function does not decay rapidly enough~\cite{HE:integral-geometry-radon-transforms, NA-mathematics-computerized-tomography}. Since our theorem is similar in spirit, we would expect a counterexample also in our case.
\NTR{Added remark about rapidly decreasing functions and added reference.}

The normal operator of the X-ray transform~$\no=\xrt^*\xrt$ is an elliptic pseudodifferential operator. Therefore it would be natural to try methods of microlocal analysis to prove our main theorem. But the usual microlocal approach does not work here in the following sense. First, if we do the identification $f\sim g$ if and only if $f-g\in C^{\infty}(\R^\dimens)$, then the claim of theorem~\ref{thm:maintheoremfornormaloperator} is not true. Namely, the assumptions $f|_V\in C^{\infty}(V)$ and $\no f|_V\in C^{\infty}(V)$ do not imply that necessarily $f\in C^{\infty}(\R^\dimens)$. Thus our result is not true modulo~$C^{\infty}$. Second, from the assumptions of theorem \ref{thm:partialdataproblem} it is clear that some of the singularities of $f$ are not visible in the data. These invisible singularities are usually difficult to reconstruct from the limited set of data ~\cite{QU-singularities-x-ray-transform-limited-data, QU-introduction-x-ray-tomography-radon, QU-artifacts-and-singularities-limited-tomography}. The surprising thing here is that even though our data is local and smooth, we can still recover a distribution.

Our theorem considers the unique continuation of the normal operator of the X-ray transform. It is then natural to ask the following question: when does the normal operator of the geodesic X-ray transform on a manifold satisfy the unique continuation property? At the moment no results are known expect in the Euclidean case. Also there does not exist any simple relationship between the normal operator and the fractional Laplacian on general manifolds. In the context of seismic applications, it would be very beneficial to generalize the result to manifolds which are equipped with a conformally Euclidean metric satisfying the Herglotz condition \cite{HE-inverse-kinematic-problem, WZ-kinematic-problem}. For example the widely used model of spherically symmetric Earth (PREM model) satisfies the Herglotz condition to a good accuracy excluding discontinuity zones~\cite{DA-PREM-model, SHE-introduction-to-seismology}. But our method of proof seems to fit only to the Euclidean case, i.e. to zero curvature. Our proof was heavily based on a density argument using polynomials and polynomials were obtained by differentiating the kernel of the Riesz potential. Our preliminary calculations suggest that we cannot obtain all the polynomials even in the constant negative curvature case. In fact the procedure fails in the very first steps: we cannot even construct polynomials of order 2. Therefore we would need a different approach if we wanted to generalize our result to non-Euclidean manifolds. 

There is another proof for theorem \ref{thm:partialdataproblem} which is based on spherical symmetry and angular Fourier series (see section~\ref{sec:alternativeproofpartialdata}). This method could perhaps generalize to some sort of spherically symmetric manifolds but it is not studied in a great detail yet. The big problem of general manifolds is that one cannot do explicit calculations. Especially we would need to express the Chebyshev polynomials in a nice form and show properties of them. The integral kernel is known in the conformally Euclidean case~\cite{deI:abel-transforms-x-ray-tomography}. However, the issue becomes to calculate the derivatives of the kernel up to any order since the idea in the alternative proof is also to obtain all the polynomials and use density.
\NTR{Minor modifications.}

In section~\ref{sec:applications} we studied the applications of our results to seismology. We discussed about a model where we measure arrival time differences of split S-waves in a thin annulus. We did a linearization of the anisotropies of the S-wave speeds in isotropic background and made an (artificial) assumption that the difference of the perturbations is independent of direction of propagation. One could also consider a more general linearization in the elastic theory. This means that we have a known isotropic elastic model and a small anisotropic perturbation in the stiffness tensor~$c_{ijkl}$ to be determined from travel time measurements. It is shown in~\cite{SHA-integral-geometry-tensor-fields} that this kind of linearization leads to the X-ray tomography problem of a tensor field of degree 4 for P-waves. For S-waves one needs to study the so-called mixed ray transform of tensor fields of degree 4. There exists a kernel characterization for the full mixed ray transform of tensors of arbitrary order on 2-dimensional compact simple Riemannian manifolds with boundary~\cite{deSZ-mixed-ray-transform}. But there are no known partial data results for the mixed ray transform. These would be highly beneficial and interesting considering applications in seismology. 

If one treats the annulus as a thin layer with respect to the radius of the Earth (``flat Earth"), the situation resembles the X-ray tomography problem in a periodic slab $[0, \epsilon]\times \mathbb{T}^2$, $\epsilon>0$.
There is a  kernel characterization for the X-ray transform of $L^2$-regular tensor fields of any order on periodic slabs of type $[0,1]\times \mathbb{T}^\dimens$ where~$\dimens$ is any non-negative integer \cite{GI-tensor-tomography-periodic-slabs}.
In particular the X-ray transform has a nontrivial kernel even for scalar fields in contrast to our result. 

\section{Integral geometry and distributions}
\label{sec:integralgeometryanddistributions}

\subsection{Distribution theory}
Let us review some basic distribution theory.
A more detailed treatment can be found in a number of introductory books on distribution theory and functional analysis, e.g.  \cite{HO-topological-vector-spaces, MI:distribution-theory, RU:functional-analysis, SA:fourier-analysis-distributions, TRE:topological-vector-spaces-distributions}.
This introduction is included for the benefit of readers less familiar with the theory and for the sake of easy reference later on. All the lemmas of this subsection are either well known or trivial and are therefore not proven.

Consider an open domain $\Omega\subset\R^\dimens$.
We denote by $\smooth(\Omega)$ the space of all smooth functions $\Omega\to\C$ and by $\csmooth(\Omega)$ the subspace consisting of compactly supported functions.
These spaces are equipped with the topology of uniform convergence of derivatives of any order on compact sets.
The topological duals of these function spaces are denoted by $\cdistr(\Omega)$ and $\distr(\Omega)$, respectively, and their elements are called distributions.
The space $\cdistr(\Omega)$ can be identified with the subspace of $\distr(\Omega)$ consisting of compactly supported distributions.

A multi-index $\beta=(\beta_1,\dots,\beta_\dimens)\in \N^\dimens$ is a $\dimens$-tuplet of natural numbers.
We use the convention that $0\in\N$.
We write $\abs{\beta}\coloneqq\beta_1+\dotso+\beta_\dimens$ and 
\begin{equation}
\partial^{\beta}=\bigg(\frac{\partial}{\partial x_1}\bigg)^{\beta_1}\cdots\bigg(\frac{\partial}{\partial x_\dimens}\bigg)^{\beta_\dimens}.
\end{equation}
The distributional derivative of order $\beta$ of $u\in\distr(\Omega)$ is defined so that
\begin{equation}
\ip{\partial^{\beta}u}{\varphi}=(-1)^{|\beta|}\ip{u}{\partial^{\beta}\varphi}
\end{equation}
for all $\varphi\in\csmooth(\Omega)$ and similarly for~$\distr$ and~$\csmooth$ replaced with~$\cdistr$ and~$\smooth$.

The value of a distribution evaluated at a test function only depends on the values of the test functions in the support of the distribution as stated in the next lemma.

\begin{lemma}
\label{lemma:testfunctionsagreeonsupport}
Let $\Omega\subset\R^\dimens$ be open and $u\in\cdistr(\Omega)$.
If $\psi_1, \psi_2\in \smooth(\Omega)$ are such that $\psi_1|_{\spt(u)}=\psi_2|_{\spt(u)}$, then $\ip{u}{\psi_1}=\ip{u}{\psi_2}$. The corresponding result also holds with~$\cdistr$ and~$\smooth$ replaced with~$\distr$ and~$\csmooth$.
\end{lemma}

It will be convenient to make a change of variables for distributions.
Let $F\colon\Omega_1\to\Omega_2$ be a $C^{\infty}$-diffeomorphism between two domains $\Omega_1,\Omega_2\subset\R^\dimens$.
The pullback $F^*u=u\circ F\in\distr(\Omega_1)$ of $u\in\distr(\Omega_2)$ is defined so that
\begin{equation}
\ip{u\circ F}{\varphi}
=
\ip{u}{(\varphi\circ F^{-1})\abs{J_{F^{-1}}}}
\end{equation}
for all $\varphi\in\csmooth(\Omega_1)$.
Here $\abs{J_{F^{-1}}}$ denotes the absolute value of the Jacobian determinant of~$F^{-1}$.
The same definition can be applied to $u\in\cdistr(\Omega_2)$ with $\varphi\in\smooth(\Omega_1)$.
The supports behave naturally under pullbacks as stated in the next lemma.

\begin{lemma}
\label{lemma:supportofdistributioninpullback}
\label{lemma:pullbackzerodistributionzero}
Let $\Omega_1, \Omega_2\subset \R^\dimens$ be open and $F\colon\Omega_1\to\Omega_2$ be a $C^{\infty}$-diffeomorphism.
If $u\in\distr(\Omega_2)$, then $\spt(u\circ F)=F^{-1}(\spt(u))$. In particular, $u=0$ if and only if $u\circ F=0$.
\end{lemma}
We will make use of the Kelvin transform or the inversion $\Kelvin\colon\R^\dimens\setminus\{0\}\to\R^\dimens\setminus\{0\}$ given by $\Kelvin(x)=\abs{x}^{-2}x$.
The Kelvin transform is its own inverse.

Any element of the spaces $\smooth(\Omega)$, $\cdistr(\Omega)$, $\csmooth(\Omega)$, and $\distr(\Omega)$ can be multiplied by an element of $\smooth(\Omega)$.
Such multiplication has an injectivity property we will need:

\begin{lemma}
\label{lemma:smoothfunctionnonzeroinsupport}
Let $\Omega\subset\R^\dimens$ be open, $u\in\cdistr(\Omega)$ and $g\in C^{\infty}(\Omega)$ such that $g\neq 0$ in $\spt(u)$.
Then $u=0$ if and only if $gu=0$.
\end{lemma}

For test functions $\varphi\in\smooth(\R^\dimens)$ we define translation~$\tau_{x_0}$ by $x_0\in\R^\dimens$ so that $(\tau_{x_0}\varphi)(x)=\varphi(x-x_0)$.
The reflection $\widetilde\varphi$ is defined by $\widetilde{\varphi}(x)=\varphi(-x)$.
Naturally $\tau_{x_0}\varphi,\widetilde\varphi\in\smooth(\R^\dimens)$.
Translations and reflections can be defined on distributions by duality.

Convolutions can also be defined for distributions (see e.g.~\cite{SA:fourier-analysis-distributions}):

\begin{lemma}
\label{lemma:convolutionbetweencompactlysupporteddistributionadntestfunction}
Let $u\in\distr(\R^\dimens)$ and $\varphi\in \csmooth(\R^\dimens)$. Then $u\ast\varphi$ has a representative $g_1\in\smooth(\R^\dimens)$ which is given by the formula $g_1(x)=\ip{u}{ \tau_x\widetilde{\varphi}}$.
Additionally, if $v\in\cdistr(\R^\dimens)$, then $v\ast\varphi$ has a representative $g_2\in\csmooth(\R^\dimens)$ which is given by the formula $g_2(x)=\ip{v} {\tau_x\widetilde{\varphi}}$.
\end{lemma}

\begin{lemma}
\label{lemma:derivativeofconvolution}
Let $u\in\cdistr(\R^\dimens)$ and $v\in\distr(\R^\dimens)$.
Then $u*v\in\distr(\R^\dimens)$ is defined via the formula
\begin{equation}
\ip{u\ast v}{\varphi}=\ip{u}{\widetilde{v}\ast\varphi}
\end{equation}
for all $\varphi\in\csmooth(\R^\dimens)$, and for every $\beta\in\N^\dimens$ the derivatives satisfy
\begin{equation}
\partial^{\beta}(u\ast v)=(\partial^{\beta}u)\ast v=u\ast(\partial^{\beta}v)
\end{equation}
in the sense of distributions.
\end{lemma}

\subsection{Integral geometry and the normal operator}
In this section we introduce basic theory of integral geometry in~$\R^\dimens$. For this we mainly follow the books~\cite{HE:integral-geometry-radon-transforms, NA-mathematics-computerized-tomography, SU:microlocal-analysis-integral-geometry}, see also \cite{RK-radon-transform-an-local-tomography}. We  define the Riesz potential~$\riesz$ and discuss about its connection to the normal operator of the X-ray transform~$\no$.
\NTR{Added more references to integral geometry.}

Denote by~$\Gamma$ the set of all oriented lines in~$\R^\dimens$. The X-ray transform of a function~$f$ is the map $\xrt f\colon\Gamma\to\R$,
\begin{equation}
\xrt f(\gamma)=\int_{\gamma}f\der s
\end{equation}
for all lines~$\gamma\in\Gamma$ assuming that the integrals exists. The integrals are finite whenever~$f$ decays fast enough at infinity. If the lines are parametrized by the set
\begin{equation}
\{(z, \theta): \theta\in S^{\dimens-1}, \ z\in\theta^{\perp}\},
\end{equation}
the X-ray transform may be written as
\begin{equation}
\xrt f(z, \theta)=\int_{\R}f(z+s\theta)\der s.
\end{equation}
It is a continuous linear map $\xrt\colon \csmooth(\R^\dimens)\to \csmooth(\Gamma)$.
The set~$\Gamma$ can be freely identified with $TS^{\dimens-1}$.
As~$\Gamma$ is a smooth manifold, the test function and distribution spaces on it can be defined similarly to the Euclidean setting.

The formal adjoint $\xrt ^*\colon \smooth(\Gamma)\to\smooth(\R^\dimens)$ is given by
\begin{equation}
\label{eq:formaladjointofxraytransform}
\xrt^*\psi (x)
=
\int_{S^{\dimens-1}}\psi(x-(x\cdot\theta)\theta, \theta)\der\theta.
\end{equation}
The function~$\psi$ can be interpreted as a function in the set of all lines.
The value~$\xrt^*\psi (x)$ is obtained by integrating~$\psi$ over all lines going through the point~$x$.
The formal adjoint does not preserve compact supports, but the integrals in its definition are taken over compact sets.
\NTR{Corrected minor typos in this section ($I\rightarrow\xrt$ and changed dimension from $n$ to $\dimens$) and added comment about the decay of $f$ in the definition of $\xrt f$.}

The operators~$\xrt$ and~$\xrt^*$ can be defined on distributions by duality.
That is, $\xrt\colon\cdistr(\R^\dimens)\to\cdistr(\Gamma)$ and $\xrt^*\colon\distr(\Gamma)\to\distr(\R^\dimens)$ are defined so that they satisfy
\begin{equation}
\ip{\xrt f}{\eta}
=
\ip{f}{\xrt^*\eta}
\end{equation}
for all $f\in\cdistr(\R^\dimens)$ and $\eta\in\smooth(\Gamma)$, and
\begin{equation}
\ip{\xrt^* g}{\varphi}
=
\ip{g}{\xrt\varphi}
\end{equation}
for all $g\in\distr(\Gamma)$ and $\varphi\in\csmooth(\R^\dimens)$. We say that $\xrt f$ vanishes on all lines which intersect an open set $V$, if $\xrt f|_{\Gamma_V}=0$ as a distribution where $\Gamma_V$ is the set of all parametrized lines intersecting $V$.
\NTR{Added explanation what vanishing of $\xrt f$ on lines which intersect $V$ means when $f$ is a distribution.}

It is often convenient to study the X-ray transform~$\xrt$ by way of its normal operator $\no=\xrt^*\xrt$.
This is not suited for all partial data scenarios and our proof in section~\ref{sec:alternativeproofpartialdata} works directly at the level of~$\xrt$, but we  make use of the normal operator elsewhere.
Due to the mapping properties established above, the normal operator maps $\no\colon\cdistr(\R^\dimens)\to\distr(\R^\dimens)$.
It is a pseudodifferential operator of order $-1$, but our problem is not well suited for a microlocal approach as discussed in section~\ref{sec:discussionofassumptionsandmethods}.
For a test function $f\in\csmooth(\R^\dimens)$ the normal operator can be expressed conveniently as~\cite{SU:microlocal-analysis-integral-geometry}
\begin{equation}
\label{eq:normaloperatorfortestfunctions}
\no f(x)
=
2\int_{\R^\dimens}\frac{f(y)}{\abs{x-y}^{\dimens-1}}\der y
=
2(f\ast\abs{\cdot}^{1-\dimens})(x).
\end{equation}
The convolution formula holds for a distribution $f\in\cdistr(\R^\dimens)$ by a duality argument, and it holds also for continuous functions which decrease rapidly enough at infinity~\cite{HE:integral-geometry-radon-transforms}.
\NTR{Added comment about functions which decrease fast enough at infinity.}

The normal operator of the X-ray transform can be inverted by the formula~\cite{SU:microlocal-analysis-integral-geometry}
\begin{equation}
\label{eq:inversionformulausingnormaloperator}
f
=
c_\dimens(-\Delta)^{1/2}\no f,
\quad c_\dimens=(2\pi\abs{S^{\dimens-2}})^{-1}
\end{equation}
for any $f\in\cdistr(\R^\dimens)$.
Here the fractional Laplacian $(-\Delta)^s$ is defined via the inverse Fourier transform $(-\Delta)^sf=\mathcal{F}^{-1}(\abs{\cdot}^{2s}\hat{f})$ and it is a non-local operator.
As can be seen in equation~\eqref{eq:inversionformulausingnormaloperator}, the normal operator of the X-ray transform is essentially~$(-\Delta)^{-1/2}$ and is inverted by~$(-\Delta)^{1/2}$.

Let $\kernel(x)=\abs{x}^{-\alpha}$ where $\alpha=\dimens-1$ or $\alpha\in\R\setminus\Z$ and $\alpha<\dimens$. We define the Riesz potential $\riesz\colon\cdistr(\R^\dimens)\to\distr(\R^\dimens)$ as
\begin{equation}
\label{eq:normaloperatorofdistribution}
\ip{\riesz f}{\varphi}
=
\ip{f\ast \kernel} {\varphi} 
\end{equation}
for all $f\in\cdistr(\R^\dimens)$ and $\varphi\in \csmooth(\R^\dimens)$.
When $\alpha<\dimens$ then $\kernel$ is locally integrable and thus defines a tempered distribution. The convolution between two distributions is well-defined when at least one of them has compact support. This implies that $\riesz f$ is always defined as a distribution when $f\in\cdistr(\R^\dimens)$.
Especially if $f\in \csmooth(\R^\dimens)$, then
\begin{equation}
\label{eq:normaloperatoroffunction}
\riesz f(x)
=
\int_{\R^\dimens}\frac{f(y)}{\abs{x-y}^{\alpha}}\der y.
\end{equation}
We call~$\kernel$ the kernel of the Riesz potential~$\riesz$.
If $\alpha=\dimens-1$, then equation~\eqref{eq:normaloperatorofdistribution} defines the normal operator of the X-ray transform~$\no$ up to a constant factor 2, see equation \eqref{eq:normaloperatorfortestfunctions}. Extensive treatment of Riesz potentials can be found in many books, see e.g. \cite{HE:integral-geometry-radon-transforms, LA-modern-potential-theory, MI-potential-theory-euclidean-spaces, SA-hypersingular-integrals}.
\NTR{Minor changes and explanation why $\riesz f$ is well-defined.}

\section{Proof of lemma~\ref{lemma:generatingpolynomials}}
\label{sec:proofofpolynomiallemma}

In this section we give a rather technical proof of lemma~\ref{lemma:generatingpolynomials}.
The proof is based on induction and algebraic relations between certain functions and their derivatives.

\begin{proof}[Proof of lemma \ref{lemma:generatingpolynomials}]
We need to show that if $\dimens\geq 2$ and $\alpha>\dimens-2$ or $\alpha\in\R\setminus\Z$, then for any polynomial~$p$ one can express $p(x\abs{x}^{-2})\abs{x}^{-\alpha}$ as a finite linear combination of derivatives of $\kernel(x)=\abs{x}^{-\alpha}$. Let us denote
\begin{equation}
A_i=x_i B, \quad  B=\abs{x}^{-2} \quad \text{and} \quad C=\abs{x}^{-\alpha}.
\end{equation}
Then one can calculate the relations
\begin{equation}
\partial_j A_i=\delta_{ij}B-2A_iA_j, \quad \abs{A}^2=B \quad \text{and} \quad \partial_i C=-\alpha A_iC.
\end{equation}
Let us also define
\begin{equation}
\label{eq:productoffirstorderterms}
D_{i_1\dotso i_n}
=
A_{i_1}\cdot\dotso\cdot A_{i_n}\cdot C=\bigg(\prod_{l=1}^n A_{i_l}\bigg)C, \quad  i_k\in\{1, \dotso, \dimens\}.
\end{equation}  
We would like to express $D_{i_1\dotso i_n}$ for all $n\in\mathbb{N}$ as a finite linear combination of derivatives of $\kernel$. The constant polynomials are given by~$\kernel$ itself. The first derivative is
\begin{equation}
\partial_i \kernel(x)=-\alpha x_i\abs{x}^{-\alpha-2}
=
-\alpha D_i.    
\end{equation}
Whence~$D_i$ can be obtained from first-order derivatives of~$\kernel$. Differentiating~$D_i$ gives
\begin{equation}
\partial_j D_i=\delta_{ij}BC-(2+\alpha)D_{ij}
\end{equation}
and the divergence is
\begin{equation}
\sum_{i=1}^\dimens\partial_i D_i
=
(\dimens-2-\alpha)BC.
\end{equation}
Combining these we obtain 
\begin{equation}
D_{ij}
=
\frac{1}{2+\alpha}\bigg(\bigg(\frac{\delta_{ij}}{\dimens-2-\alpha}\sum_{i=1}^\dimens\partial_iD_i\bigg)-\partial_jD_i\bigg).
\end{equation}
We have thus expressed the terms~$D_{ij}$ as a finite linear combination of the derivatives of the terms~$D_i$ which were multiples of the first-order derivatives of~$\kernel$. Hence~$D_{ij}$ can be expressed as a finite linear combination of second-order derivatives of~$\kernel$.

We claim that~$D_{i_1\dotso i_n}$ is a finite linear combination of~$n$th order derivatives of~$\kernel$ for all $n\in\N$ and we have shown this for $n=0, 1, 2$. The lemma follows from this claim. Let us assume that the claim holds for some $m-1\in\mathbb{N}$. Then $D_{i_1\dotso i_{m-1}}$ is a finite linear combination of $(m-1)$th order derivatives of~$\kernel$. Thus $\partial_{i_m}D_{i_1\dotso i_{m-1}}$ is a finite linear combination of~$m$th order derivatives of~$\kernel$ and a calculation shows that
\begin{equation}
\label{eq:derivativeofd}
\partial_{i_m}D_{i_1\dotso i_{m-1}}=(2-2m-\alpha)D_{i_1\dotso i_m}+\sum_{j=1}^{m-1}\bigg(\delta_{i_m i_j}BC \prod_{\substack{l=1 \\ l\neq j}}^{m-1}A_{i_l}\bigg).
\end{equation}
Let us then calculate the divergence from equation~\eqref{eq:derivativeofd}. We get  
\begin{equation}
\label{eq:divergence}
\sum_{i_k=1}^\dimens\partial_{i_k}D_{i_1\dotso i_k\dotso i_{m-1}}
=
(\dimens-m-\alpha)BC  \ \prod_{\substack{l=1 \\ l\neq k}}^{m-1}A_{i_l}.
\end{equation}
From equations~\eqref{eq:derivativeofd} and~\eqref{eq:divergence} we obtain the following expression for~$D_{i_1\dotso i_m}$
\begin{align}
\frac{1}{2-2m-\alpha}\Bigg(\partial_{i_m}D_{i_1\dotso i_{m-1}}-\frac{1}{\dimens-m-\alpha}\sum_{j=1}^{m-1}\bigg(\delta_{i_m i_j}\sum_{i_j=1}^\dimens\partial_{i_j}D_{i_1\dotso i_j\dotso i_{m-1}}\bigg)\Bigg)
\end{align}  
which is by the induction assumption a finite linear combination of~$m$th order derivatives of~$\kernel$. Thus the claim follows for all $n\in\mathbb{N}$.
\end{proof}

\section{Alternative proofs of the main theorems}
\label{sec:alternativeproofs}
\NTR{Added motivation for this section, added new proofs of unique continuation property, other minor modifications.}

In this section we give alternative proofs to our main theorems, theorem \ref{thm:maintheoremfornormaloperator} and theorem \ref{thm:partialdataproblem}. We believe that presenting several proofs opens more possibilities to generalize the results and gives more tools for solving similar unique continuation problems and partial data problems. 

We prove theorem \ref{thm:maintheoremfornormaloperator} under the stronger assumption $\riesz f|_V=0$ for a slightly larger class of distributions, i.e. rapidly decreasing distributions. We do it in two alternative ways. First proof is based on convolution approximation and density of polynomials. The second approach uses the unique continuation property of the fractional Laplacian. The second proof is short since it relies on a strong result. The unique continuation of $(-\Delta)^s$, $s\in (0, 1)$, is based on technical results about Carleman estimates and Caffarelli-Silvestre extensions~\cite{GSU-calderon-problem-fractional-schrodinger}. 

We then prove theorem \ref{thm:partialdataproblem} first for compactly supported smooth functions using angular Fourier series and density argument based on differentiation of an integral kernel. By a standard mollification argument we obtain the same result for compactly supported distributions. The proof works directly at the level of the X-ray transform and does not use the normal operator at all. Therefore we do not need to use any unique continuation results in the proof of the partial data problem.

We briefly go through our notations. We denote by $\slowly(\R^\dimens)$ the space of polynomially increasing smooth functions, by~$\tempered(\R^\dimens)$ the space of tempered distributions, by~$\rapidly(\R^{\dimens})$ the space of rapidly decreasing distributions and by $H^r(\R^\dimens)$ the fractional $L^2$-Sobolev space of order $r\in\R$. For precise definitions see \cite{HO-topological-vector-spaces, ML-strongly-elliptic-systems, SA:fourier-analysis-distributions, TRE:topological-vector-spaces-distributions}. For us it is enough to know that $\cdistr(\R^\dimens)\subset\rapidly(\R^\dimens)\subset\tempered(\R^\dimens)$ and 
$$
\rapidly(\R^\dimens)\subset\bigcup_{r\in\R}H^r(\R^\dimens).
$$
Rapidly decreasing continuous functions, i.e. continuous functions which decrease faster than any polynomial at infinity, are contained in~$\rapidly(\R^\dimens)$. The convolution operator~$\ast$ is a separately continuous map $\ast\colon\rapidly(\R^\dimens)\times\tempered(\R^\dimens)\rightarrow\tempered(\R^\dimens)$. This implies that the Riesz potential $\riesz f=f\ast\abs{\cdot}^{-\alpha}$ is defined as a distribution when $f\in\rapidly(\R^\dimens)$ and $\alpha<\dimens$. The Fourier transform is a bijective map from $\rapidly(\R^\dimens)$ onto $\slowly(\R^\dimens)$ and the usual convolution formula $\widehat{f\ast g}=\widehat{f}\cdot\widehat{g}$ holds in the sense of distributions when $f\in\rapidly(\R^\dimens)$ and $g\in\tempered(\R^\dimens)$. 

\subsection{Using convolution approximation}

In this section we prove theorem \ref{thm:maintheoremfornormaloperator} under the assumption $\riesz f|_V=0$ first for Schwartz functions. The result follows also for rapidly decreasing distributions by considering the mollifications $f\ast j_{\epsilon}$.

\begin{theorem}
\label{thm:alternativeconvolutionproof}
Let $\alpha=\dimens-1$ or $\alpha\in\R\setminus\Z$ and $\alpha<\dimens$. Let $f\in\rapidly(\R^\dimens)$ and $V\subset\R^\dimens$ any nonempty open set. If $f|_V=\riesz f|_V=0$, then $f=0$.
\end{theorem}

\begin{proof}
We can assume that $0\in V$. Let first $f\in\schwartz(\R^\dimens)$. Like in the proof of theorem \ref{thm:maintheoremfornormaloperator} we smoothen the kernel $\kernel$ near the origin, let this smoothened version be $g\in C^{\infty}(\R^\dimens)$. There is $\epsilon>0$ such that $(f\ast g)|_{B(0, \epsilon)}=(f\ast\kernel)|_{B(0, \epsilon)}$. It holds that $\partial^{\beta}(f\ast g)=f\ast\partial^{\beta}g$ where by lemma \ref{lemma:generatingpolynomials} one obtains all the polynomials $p$ in the form $p(K(x))\kernel(x)$ by taking finite linear combinations of $\partial^{\beta}g$. Since $f$ is not supported in a ball $B$ centered at the origin, we can use the Kelvin transform to obtain
\begin{align*}
0=\int_{B^c}f(y)p(y\abs{y}^{-2})\abs{y}^{-\alpha}\der y=\int_{\widetilde{B}\setminus \{0\}}f(x\abs{x}^{-2})p(x)\abs{x}^{\alpha}\abs{J_K(x)}\der x
\end{align*}
where $\widetilde{B}$ is some closed ball centered at the origin. One can calculate that $\abs{J_K(x)}=\abs{x}^{-2\dimens}$ (see \cite[Remark 4.2]{GSU-calderon-problem-fractional-schrodinger}). Since $f$ goes rapidly to zero at infinity, we can extend the function $x\mapsto f(x\abs{x}^{-2})\abs{x}^{\alpha}\abs{J_K(x)}$ continuously to zero and we call this extension~$\widetilde{f}$. We obtain
$$
\int_{\widetilde{B}}\widetilde{f}(x)p(x)\der x=0
$$
for all polynomials $p$. Since $\widetilde{f}$ is continuous and $\widetilde{B}$ is compact, by the Stone-Weierstrass theorem $\widetilde{f}=0$. This implies $f=0$.

Then let $f\in\rapidly(\R^\dimens)$. Denote by $j_{\epsilon}\in\csmooth(\R^\dimens)$ the standard mollifier and consider the mollifications $f_{\epsilon}=f\ast j_{\epsilon}\in\schwartz(\R^\dimens)$. Since $\riesz(f\ast j_{\epsilon})=\riesz f\ast j_{\epsilon}$ it follows that $f_{\epsilon}|_W=\riesz f_{\epsilon}|_W=0$ for small enough $\epsilon>0$ and $W\subset V$ open. By the first part of the proof $f_{\epsilon}=0$ for small $\epsilon>0$. This implies $f=0$ since $f_{\epsilon}\rightarrow f$ as distributions in $\tempered(\R^\dimens)$ when $\epsilon\rightarrow 0$.
\end{proof}

We remark that theorem \ref{thm:alternativeconvolutionproof} implies uniqueness for the partial data problem (theorem \ref{thm:partialdataproblem}) when $f$ is a continous function which decreases faster than any polynomial. We can thus relax the assumption of compact support to rapid decay at infinity in theorem \ref{thm:partialdataproblem}.

\subsection{Using unique continuation of the fractional Laplacian}
\label{sec:alternativeproofrieszpotential}

Here we give an alternative proof for a modified version of theorem \ref{thm:maintheoremfornormaloperator} using Fourier analysis and unique continuation of $(-\Delta)^s$ in $H^{r}(\R^\dimens)$, $r\in\R$, when $0<s<1$. The unique continuation of~$(-\Delta)^s$ is proved in \cite{GSU-calderon-problem-fractional-schrodinger}.

\begin{theorem}
\label{thm:alternativeproofrieszpotential}
Let $f\in\rapidly(\R^{\dimens})$, $V\subset\R^{\dimens}$ any nonempty open set and $0<\alpha<\dimens$ such that $(\alpha-\dimens)/2\not\in\Z$. If $f|_V=0$ and $\riesz f|_V=0$, then $f=0$.
\end{theorem}

\begin{proof}
There is $k\in\N$ such that $-k<(\alpha-\dimens)/2<-k+1$. Using the convolution property of the Fourier transform we can write
\begin{align*}
\riesz f
=
f\ast\abs{\cdot}^{-\alpha}
=
c_\dimens\mathcal{F}^{-1}(\mathcal{F}(f\ast\abs{\cdot}^{-\alpha}))
=
c_\dimens\mathcal{F}^{-1}(\hat{f}\abs{\cdot}^{\alpha-\dimens})
=
c_\dimens(-\Delta)^{\frac{\alpha-\dimens}{2}}f,
\end{align*}
where $c_\dimens>0$ is a constant depending on dimension. Since $(-\Delta)^{\frac{\alpha-\dimens}{2}}f$ is a tempered distribution, again by the properties of the Fourier transform it follows that $(-\Delta)^k(-\Delta)^{\frac{\alpha-\dimens}{2}}f=(-\Delta)^{k+\frac{\alpha-\dimens}{2}}f=(-\Delta)^s f$ where $s=k+(\alpha-\dimens)/2\in (0, 1)$. Since $(-\Delta)^k$ is a local operator and $(-\Delta)^{\frac{\alpha-\dimens}{2}}f$ vanishes in the open set $V$, we obtain the conditions $f|_V=0$ and $(-\Delta)^s f|_V=0$. Now $f\in\rapidly(\R^\dimens)$ which implies $f\in H^r(\R^\dimens)$ for some $r\in\R$. By \cite[Theorem 1.2]{GSU-calderon-problem-fractional-schrodinger} we obtain $f=0$.
\end{proof}

We remark that theorem \ref{thm:alternativeproofrieszpotential} implies the unique continuation of the normal operator of the X-ray transform in dimensions $\dimens\geq 2$ since in that case $0<\dimens-1=\alpha<\dimens$ and $(\alpha-\dimens)/2=-1/2\not\in\Z$. 

\subsection{Angular Fourier series approach}
\label{sec:alternativeproofpartialdata}

In this section we give another proof of theorem~\ref{thm:partialdataproblem}. We assume without loss of generality that $f$ is supported in $\overline{B}(0, R^{\prime})\setminus B(0, R)$ for some $R^{\prime}>R>0$ and that $0\in V$.
The proof is based on a similar idea as before, differentiation of an integral kernel and density of polynomials.
However, now we study the X-ray transform directly and exploit the underlying spherical symmetry by using angular Fourier series expansion.
\NTR{Generalized this proof to compactly supported distributions by convolution approximation, corrected the notation $I\rightarrow \xrt$ and other minor modifications.}

In the next theorem, when $f\in C_c(\R^\dimens)$ it would be enough to assume that the X-ray transform~$\xrt f$ vanishes to infinite order on all lines through the origin, i.e. $\partial_r^n (\xrt f)(r, \theta)|_{r=0}=0$ for all $n\in\N$. This is a similar assumption that we used in theorem \ref{thm:maintheoremfornormaloperator}.
\NTR{Clarified what vanishing to infinite order on all lines through the origin means for $\xrt f$.}

\begin{theorem}
\label{thm:alternativemaintheorem}
Fix any $0<\epsilon<R<R^{\prime}$. Let $f\in\cdistr(\R^\dimens)$ such that $\spt(f)\subset\overline{B}(0,R^{\prime})\setminus B(0,R)$. If $f$ integrates to zero over all lines in $B(0,R^{\prime})$ that meet $B(0,\epsilon)$, then $f=0$.
\end{theorem}

\begin{proof}
Without loss of generality we can assume that $R^{\prime}=1$. Let first $f\in C_c(\R^\dimens)$. By intersecting the origin with 2-planes it is enough to prove the result in two dimensions. The function~$f$ can be expressed as an angular Fourier series
\begin{equation}
\label{eq:angularfourierexpansion}
f(r,\theta)
=
\sum_{k\in\Z}e^{ik\theta}a_k(r).
\end{equation}
Our goal is to show that $a_k=0$ for all $k\in\Z$. When we parameterize the lines in~$\R^2$ by their closest point to the origin and use polar coordinates for these points, we find 
\begin{equation}
\xrt f(r,\theta)
=
\sum_{k\in\Z}e^{ik\theta}\A_{\abs{k}}a_k(r),
\end{equation}
where~$\A_k$ is the generalized Abel transform defined by
\begin{equation}
\label{eq:abeltransform}
\A_kg(z)
=
2\int_z^1K_k(z,y)g(y)\der y.
\end{equation}
Here the kernel is $K_k(z,y)=T_k(z/y)[1-(z/y)^2]^{-1/2}$ and $T_k$ are the Chebyshev polynomials.

We know that $f(r,\theta)=0$ when $r<R$ and $\xrt f(r,\theta)=0$ when $r<\eps$. For the Fourier components~$a_k(r)$ this means that for every $k\in\Z$ we have $a_k(r)=0$ for $r<R$ and $\A_ka_k(r)=0$ for $r<\eps$. Hence we get
\begin{equation}
\label{eq:kernelintegraliszero}
\int_R^1K_k(z,y)a_k(y)\der y
=
0
\end{equation}
for every $z\in[0,\eps)$. Like in the proof of theorem~\ref{thm:maintheoremfornormaloperator}, we differentiate the integral kernel~$n$ times in~\eqref{eq:kernelintegraliszero} with respect to~$z$ and evaluate at~$z=0$ to obtain
\begin{equation}
\label{eq:derivativesofthekernelzero}
\int_R^1 D_k^n(y)a_k(y)\der y
=
0
\end{equation}
for all $n\in\N$ and $k\in\Z$, where $D_k^n(y)=\partial_z^n K_k(z,y)|_{z=0}$.

By scaling arguments  $D_k^n(y)=A_k^ny^{-n}$ for some numbers~$A_k^n$. The term $k=0$ is
\begin{equation}
\label{eq:zerothtermofderivative}
A_0^n
=
\begin{cases}
(n-1)!!^2,&n\text{ even},\\
0,&n\text{ odd}.
\end{cases}
\end{equation}
We denote the coefficient of~$x^l$ in~$T_k(x)$ by~$t_k^l$.
The~$l$th derivative of~$T_k(x)$ at~$x=0$ is $l!t_k^l$. The coefficients also satisfy
\begin{equation}
\label{eq:chebyshevcoefficients}
\sum_{l=0}^kt_k^l
=
T_k(1)
=
1.
\end{equation}
By basic properties of Chebyshev polynomials $t_k^l=0$ if $l-k$ is odd or $l>k$. Using $K_k(z,y)=T_k(z/y)K_0(z,y)$ and the product rule of higher order derivatives we find
\begin{equation}
A_k^n
=
\sum_{l=0}^n
{n\choose l}l!t_k^lA_0^{n-l}.
\end{equation}
By parity properties it is clear that~$A_k^n$ vanishes unless both~$n$ and~$l$ are even or both are odd. 

We will show that for any $k\in\N$ there is a number~$N(k)$ so that $A_k^n>0$ when $n\geq N(k)$ and parity is right. For $k=0$ this follows from equation~\eqref{eq:zerothtermofderivative} with $N(k)=0$. Consider first the case when~$n$ and~$k$ are both even and assume $n>k$. A calculation shows that
\begin{equation}
A_k^n=n!
\frac{(n-1)!!}{n!!}
\sum_{m=0}^{k/2}
\left[
t_k^{2m}
+
t_k^{2m}
\left(
\frac{(n-2m-1)!!n!!}{(n-2m)!!(n-1)!!}
-1
\right)
\right].
\end{equation}
There are only finitely many terms in the sum, and for every~$m$ we have
\begin{equation}
\lim_{n\to\infty}
\frac{(n-2m-1)!!n!!}{(n-2m)!!(n-1)!!}
=
1.
\end{equation}
Equation~\eqref{eq:chebyshevcoefficients} implies $\sum_{m=0}^{k/2}t_k^{2m}=1$ so that
\begin{equation}
\lim_{n\to\infty}
\sum_{m=0}^{k/2}
\left[
t_k^{2m}
+
t_k^{2m}
\left(
\frac{(n-2m-1)!!n!!}{(n-2m)!!(n-1)!!}
-1
\right)
\right]
=
1.
\end{equation}
Therefore $A_k^n>0$ for sufficiently large~$n$ as claimed. Similarly one can show for odd indices that
\begin{equation}
A_k^n=n!
\frac{(n-2)!!}{(n-1)!!}
\sum_{m=0}^{(k-1)/2}
\left[
t_k^{2m+1}
+
t_k^{2m+1}
\left(
\frac{(n-2m-2)!!(n-1)!!}{(n-2m-1)!!(n-2)!!}
-1
\right)
\right].
\end{equation}
With the same limit argument we get $A_k^n>0$ for large~$n$.

We fix any $k\in\Z$ and use~\eqref{eq:derivativesofthekernelzero} to show that $a_k=0$. By symmetry it suffices to consider $k\geq 0$. We found~$N(k)$ so that $A_k^n\neq 0$ for $n\geq N(k)$ when $n-N(k)$ is even. We find
\begin{equation}
\int_R^1 y^{-N(k)-2m}a_k(y)\der y
=
0
\end{equation}
for every $m\in\N$.
By linearity 
\begin{equation}
\int_R^1 y^{-N(k)}p(y^{-2})a_k(y)\der y
=
0
\end{equation}
for any polynomial~$p$.
Changing variable to~$s=y^{-2}$ and defining new coefficients $\tilde a_k(s)=s^{N(k)/2-3/2}a_k(s^{-1/2})$, we obtain
\begin{equation}
\int_1^{R^{-1/2}} p(s)\tilde a_k(s)\der s
=
0.
\end{equation}
By density of polynomials $\tilde a_k(s)=0$ for all $s\in[1, R^{-1/2}]$.
This implies  $a_k=0$ for all $k\in\Z$ and hence $f=0$.

Then let $f\in\cdistr(\R^\dimens)$ and consider the mollifications $f\ast j_{\epsilon}\in \csmooth(\R^\dimens)$. Following Helgason \cite{HE:integral-geometry-radon-transforms} we define the ``convolution'' 
$$
(g\times\varphi)(z, \theta)=\int_{\R^\dimens}g(y)\varphi(z-y, \theta)\der y
$$
where $g\in\csmooth(\R^\dimens)$ and $\varphi\in\csmooth(\Gamma)$. By a simple calculation one can show that $\xrt^*(g\times\varphi)=g\ast\xrt^*\varphi$. Using the properties of the convolutions $\ast$ and $\times$ we obtain 
\begin{align*}
\ip{\xrt(f\ast j_{\epsilon})}{\varphi}=\ip{f\ast j_{\epsilon}}{\xrt^*\varphi}=\ip{f}{j_{\epsilon}\ast \xrt^*\varphi}=\ip{f}{\xrt^*(j_{\epsilon}\times\varphi)}=\ip{\xrt f}{j_{\epsilon}\times\varphi}.
\end{align*}
Thus for small enough $\epsilon>0$ and $\tilde{R}>0$ we get that $(f\ast j_{\epsilon})|_{B(0, \tilde{R})}=0$ and $\xrt(f\ast j_{\epsilon})$ vanishes on all lines which intersect $B(0, \epsilon)$. The first part of the proof implies $f\ast j_{\epsilon}=0$ for small $\epsilon>0$. The claim follows since $f\ast j_{\epsilon}\rightarrow f$ in $\cdistr(\R^\dimens)$ when $\epsilon\rightarrow 0$.
\end{proof}

We remark that the assumption that $f$ is supported away from the origin is crucial since it turns a Volterra integral equation into a Fredholm integral equation. This simplifies the derivatives of expression \eqref{eq:abeltransform}.

\bibliography{refs} 
\bibliographystyle{abbrv}

\end{document}